\documentclass[a4paper,12pt]{article}

\usepackage{ulem}

\usepackage{graphicx,color}

\usepackage[latin1]{inputenc}
\usepackage{amsfonts}
\usepackage{amsmath,amsthm}
\usepackage{amssymb}
\usepackage{verbatim}
\usepackage{latexsym}

\newcommand{\Rd}{\mathbb R^d}
\newcommand{\R}{\mathbb R}
\newcommand{\Z}{\mathbb Z}
\newcommand{\N}{\mathbb N}
\newcommand{\E}{\mathbb E}
\newcommand{\ind}{{\bf 1}}
\newcommand{\rg}{\textrm{rank\;}}
\newcommand{\A}{\textrm{{\bf A}}}
\newcommand{\pa}{\textrm{{\bf PA}}}
\newcommand{\na}{\textrm{{\bf NA}}}

\newcommand{\cov}{\mbox{\rm Cov}}
\newcommand{\var}{\mbox{\rm Var}}
\newtheorem{theorem}{Theorem}[section]
\newtheorem{lemma}[theorem]{Lemma}
\newtheorem{lemma*}{Lemma}

\newtheorem{corollary}[theorem]{Corollary}
\newtheorem{definition}[theorem]{Definition}
\newtheorem{definition*}[lemma*]{Definition}
\newtheorem{example}[theorem]{Example}

\newtheorem{remark}[theorem]{Remark}
\newcommand{\eqd}{\stackrel{d}{=}}
\newcommand{\tod}{\stackrel{d}{\longrightarrow}}
	\newcommand{\ud}{\mathrm{d}}
	\newcommand{\Int}{\int\limits}

\title{Long range dependence of heavy tailed random functions}
\author{Rafal Kulik, Evgeny Spodarev}

\begin{document}

\maketitle

\begin{abstract}
We introduce a  definition of long range dependence of random processes and fields on an (unbounded) index space $T\subseteq \R^d$ in terms of integrability of the covariance of indicators that a random function exceeds any given level. This definition is particularly designed to cover the case of random functions with infinite variance. We show the value of this new definition and its connection to limit theorems on some examples including subordinated Gaussian as well as random volatility fields and time series.

AMS Subj. Class.: Primary 60G10; Secondary 60G60, 60G15, 60F05.
\end{abstract}
\section{Introduction}

Let $X=\{ X_t, t\in T \}$ be a stationary random field on an unbounded index subset $T$ of $\Rd$, $d\ge 1$, defined on an abstract probability space $(\Omega, {\cal F}, P)$.
If $X_0$ is square integrable
then the classical definition of long range dependence is
\begin{equation}\label{lrdfv}
\int_T |C_X(t)|\, dt=+\infty,
\end{equation}
 where $C_X(t)=\cov(X_0, X_t)$, $t\in T$. There are also other definitions e.g. in terms of spectral density of $X$ being unbounded at zero, growth comparison of partial sums (Allan sample variance), the order of the variance of sums  going to infinity, etc., see the modern reviews in \cite{GiraitisKoulSurg12}, \cite{beran:kulik:2013}, \cite{Samorod16} for processes and  \cite{Lavancier06} for random fields.  All these approaches are not equivalent to each other.

 More importantly,  there is no unified approach to define long memory property if $X$ is heavy tailed, that is with infinite variance. Many authors use the phenomenon of phase transition in certain parameters of the field (such as stability index, Hurst index, heaviness of the tails, etc.) regarding their different limiting behaviour. To give just a few examples, we mention
 \cite{sly:heyde:2008} for the subordinated heavy-tailed Gaussian time series whereas  \cite{Samorod04}, \cite{RoySam08}, \cite{Roy10}, \cite{Owada2015MaximaOL}, \cite{samorodnitsky2019} consider the extreme value behaviour of partial maxima of stable random processes and fields and a connection with their ergodic properties. In \cite[p. 76]{DehlPhil02}, the short or long memory for stationary time series is defined by using different limits in functional limit theorems.
 Papers \cite{DamPaul17,Paul16} analyze different measures of dependence (such as $\alpha$-spectral covariance) for linear random fields with infinite variance lying in the domain of attraction of a stable law. Those are used to define various types of memory and prove corresponding limit theorems for partial sums.


The main goal of our paper is to give a simple uniform view into long range dependence which applies to any stationary (light or heavy tailed) random field $X$; see Definition \ref{def:lrd}.
In Section \ref{subsec:CheckLRD} we show that all rapidly mixing random fields  are short range dependent in the sense of the new definition. No moment assumptions are needed there.
In Section \ref{subsec:SubGaussian}, the sufficient conditions for a subordinated Gaussian (possibly heavy-tailed) random field to be short or long range dependent are given.
We show that the transition from short to long memory occurs at the same boundary for both finite and infinite variance random fields; see Theorem \ref{thm:subGaussSRD} and Example \ref{ex:e_x2}. This cannot be achieved using the classical definitions based on second-order properties.  
In the next section, the same is done for stochastic volatility random fields of the form
$X_t=G(Y_t)Z_t$. Different sources of long range dependence are described. Conditions for long or short memory of $\alpha$--stable moving averages and certain max--stable processes are discussed in the forthcoming paper \cite{Makoginetal19}.

As indicated above, one can approach long memory from two different perspectives: through the distributional properties of the process or limiting behaviour of suitable statistics. Our definition falls into the first category. Thus, as the next step, we attempt to link the definition with limit theorems. In this context,  
the appropriate statistic to study appears to be the volume of level sets of the field.
This is done in Section \ref{sec:lt}. First, we consider subordinated Gaussian random fields and show the agreement between our definition and the limiting behaviour. See Section \ref{subsect:LTInfVar}. In the following section we indicate that our definition is not suitable to capture limiting behaviour of the empirical mean.  
In Section \ref{subsec:LTInt} we consider the corresponding problems for random volatility models. In order to do so, we have to develop limiting theory for integral functionals of random volatility models, including the case of limit theorems for the volume of level sets of $X$. These results are of independent interest.

For better readability, proofs of the most of results are moved to Appendix.

\section{Preliminaries}\label{sec:Prelim}

Recall that $T$ is an unbounded subset of $\mathbb{R}^d$.
Let $\N_0=\N\cup \{ 0 \}$, and let $\nu_d(\cdot)$ be the $d$--dimensional Lebesgue measure. We denote by $\R_\pm$ either $\R_+=[0,+\infty)$ or $\R_-=(-\infty,0]$, depending on the context. For instance, $G:\R\to \R_\pm$ means that $G$ maps $\R$ either to $\R_+$ or to $\R_-$.  Let $\|\cdot\|$ be a norm in the Euclidean space $\Rd$.
For two functions $f,g: \R\to \R $  we write   $f(x) \sim  g(x)$, $x\to a$ if $\lim_{x\to a} f(x)/g(x)=1$, where $g(x)\neq 0$ in a neighbourhood of $a$.
Let $\langle f, g\rangle = \int_{\R} f(x)g(x) \, dx$ be the inner product in the space $L^2(\R)$ of square integrable functions. Additionally, we shall make use of the inner product
$\langle f, g\rangle_{\varphi} = \int_{\R} f(x)g(x)\varphi(x) \, d x$ in the space $L^2_{\varphi}(\R)$ of functions which are square integrable with the weight $\varphi$, where $\varphi$ is the standard normal density. For a finite measure $\mu$ on $\R$, let $\textrm{supp} (\mu)$ be its support, i.e., the compliment of the largest measurable subset of $\mu$-measure zero in $\R$.

Let $(\Omega, \mathcal{F}, P)$ be a probability space. We say that $\{ X_t, \; t\in T \}$ is a {\it white noise} if it consists of i.i.d. random variables $X_t$.

{For any random variable $X$ let} $F_X(x)=P(X\le x)$ and $\bar{F}_X(x)=1-F_X(x)$ be the cumulative distribution function and the tail distribution function of $X$, respectively.  Let $F_{X,Y}(x,y)=P(X\le x, Y\le y)$, $x,y\in \R$ be the bivariate distribution function of a random vector $(X,Y)$. Later on we make use of the known formula
\begin{equation} \label{eq:condcov}
\cov(X,Y)=\E\left( \cov(X,Y|{\cal A}) \right)+ \cov \left( \E(X|{\cal A}), \E(Y|{\cal A}) \right)
\end{equation}
for any $\sigma$--algebra ${\cal A}\subset {\cal F}$.


 A random field $X=\left\{X_t,t\in T\right\}$  is called \emph{associated} (\A)
 if
 $$\cov\left(f\left(X_I\right),g\left(X_I\right)\right)) \ge 0 $$ for any finite subset $I\subset T$ and for any
bounded coordinatewise non--decreasing Borel functions $f,g:\R^{|I|}\rightarrow\R$, where $X_I=\{
X_t,t\in I \}$. $X$ is called \emph{ positively associated (\pa)} or \emph{negatively associated (\na)} if
$$\cov\left(f\left(X_I\right),g\left(X_J\right)\right)) \ge 0 \ (\le 0),$$ respectively for all finite disjoint subsets $I,J\subset T$, and for any
bounded coordinatewise non--decreasing Borel functions $f:\R^{|I|}\rightarrow\R$, $g:\R^{|J|}\rightarrow\R$, see e.g. \cite{BulSha07}.

We use the notation $B\sim S_{\alpha} \left(\sigma,1,0  \right)$ for an  $\alpha$-stable subordinator $B$  with scale parameter
$\sigma>0$, cf. \cite{SamorodTaqqu94}.
\section{Long range dependence}
\label{sec:lrd}

Consider a real--valued stationary random field $X=\{X_t,t\in T\}$. Introduce
$$
\cov_X(t,u,v)=\cov \left(\ind(X_0>u),\ind(X_t>v) \right), \quad t\in T, \, x,v\in \R.
$$
It is always defined as the indicators involved are bounded functions.
\begin{definition}\label{def:lrd}
A random field $X=\{X_t,t\in T\}$ is called short range dependent (s.r.d.) if for any finite measure $\mu$ on $\R$
$$
\sigma^2_{\mu,X}:=\int\limits_T\int\limits_{\R^2}|{\cov}_X(t,u,v)|\, \mu(du)\, \mu(dv)\,dt<+\infty.
$$
$X$ is
long range dependent (l.r.d.)  if there exists a finite measure $\mu$ on $\R$ such that
$\sigma^2_{\mu,X}=+\infty.$
For discrete parameter random fields (say, if $T\subseteq \Z^d$), the $\int_T \,dt$ above  should be replaced by   $\sum_{t\in T: t\neq 0}$.
\end{definition}

\subsection{Motivation} \label{subsec:Motivation}
Assume that $X$ is
stationary with marginal distribution function  $F_X(x)=P(X_0\le x)$, $x\in\R$,   covariance function $C(t)=\cov(X_0, X_t)$, $t\in T$, and moreover,
\begin{equation} \label{eq:property}
\cov_X(t,u,v)\ge 0 \mbox{ or } \le 0  \mbox{ for all } t\in T, \, u,v\in\R.
\end{equation}
Examples of $X$ with this property are all \pa \  or \na -  random functions. Applying  \cite[Lemma 2]{Lehmann66}, we have (the equality is originally attributed to Hoeffding (1940))
\begin{equation} \label{eq:Hoeffding}
C_X(t)=\int\limits_{\R^2} \cov_X(t,u,v) \, du\,dv.
\end{equation}
Then, $X$ is long range dependent if
\begin{equation}\label{eq:cov-finite-variance}
\int\limits_T |C_X(t)|\, dt=\int\limits_T\int\limits_{\R^2}|\cov_X(t,u,v)|\, du\,dv\,dt=+\infty,
\end{equation}
which agrees with the classical definition.

However, in Definition \ref{def:lrd} we integrate $|\cov_X(t,u,v)|$ with respect to a finite measure $\mu\times\mu$  instead of Lebesgue measure $du \, dv$.
First, in case of the infinite variance the right-hand side in (\ref{eq:cov-finite-variance}) is often infinite, regardless of a dependence structure. As such, the classical definition of long memory is irrelevant in the infinite variance case. Second, our definition
will have a natural link with
the asymptotic behavior of volumes of excursions of $X$ above levels $u,$ $v$.
Recall the functional central limit theorem (CLT) for normed volumes of excursion sets of $X$ at level $u$ proven in  \cite{MeshShash11} (see also \cite[Theorem 9, p. 234]{Spodarev14} for a generalization of this result to fields without a finite second moment).
Namely,    for a large class of weakly dependent stationary random fields $X$ on $\Rd$, the function
$$\int\limits_{\Rd}\cov_X(t,u,v)\,dt\;, \ \ u,v\in\R$$
is the covariance function of the centered Gaussian process which appears as a limit of
\begin{equation}\label{eq:normedLevelSet}
\frac{\nu_d\left(\{  t\in [0,n]^d: X_t>u \}\right) -n^d \bar{F}_{X}(u)  }{n^{d/2}}, \quad u\in\R,\quad n\to\infty
\end{equation}
in $\cal{D}(\R)$ equipped with the $J_1$ Skorokhod topology.
If in particular the random field is \pa \  or \na, then by the  continuous mapping theorem, it holds
\begin{equation}\label{eq:normedLevelSetmu}
\frac{\int_{\R} \nu_d\left(\{  t\in [0,n]^d: X_t>u \}\right) \mu(du) -n^d \int_{\R} \bar{F}_{X}(u)\mu(du)  }{n^{d/2}}\tod N(0,\sigma^2_{\mu,X})
\end{equation}
as $n\to\infty$ for any finite measure $\mu$ with $\sigma^2_{\mu,X} $ as in Definition \ref{def:lrd}.
So $X$ is
s.r.d. if the asymptotic covariance $\sigma^2_{\mu,X} $ in the central limit theorem \eqref{eq:normedLevelSetmu} is finite for any finite integration measure $\mu$ prescribing the choice of levels $u$. 
On the contrary, 
\begin{equation}\label{eq:sigma^2}
\sigma^2_{\mu,X}=+\infty
\end{equation}
for $\mu=\delta_{\{ u_0 \}}$ means
that a different normalization is needed in (\ref{eq:normedLevelSet}) and a non-Gaussian limit may arise.
%

Let us point out at a possible interpretation of Definition \ref{def:lrd}  in financial context.   Assume $X=\{ X_t, t\in \Z\}$ to be a time series representing the stock price for which an American option at price $u_0>0$, $t\in [0,n]$, $n\in\N$ is issued.   The customer may buy a call at price $u_0$ whenever $X_t>u_0$ for some $t\in[0,n]$.  Relation \eqref{eq:normedLevelSetmu}  with  $\mu=\delta_{\{ u_0 \}}$ writes here
$$
\frac{ \nu_1\left(\{  t\in [0,n]: X_t>u_0 \}\right) -n \bar{F}_{X}(u_0)  }{\sqrt{n}}\tod N(0,\sigma^2_{\delta_{\{ u_0 \}},X}).
$$
Then the long range dependence in the sense of Definition \ref{def:lrd} of the stock price $X$  (i.e., $\sigma^2_{\delta_{\{ u_0 \}},X} =+\infty$) means that the amount of time within $[0,n]$ at which the option may be exercised is not asymptotically normal for large time horizons $n$.  On the contrary, the s.r.d. of stock $X$ means asymptotic normality of this time span for {\it any}  price $u_0$ for which the option was issued provided that $X$ satisfies conditions of papers \cite{MeshShash11} or \cite{Spodarev14}.

In terms of potential theory, the value $\sigma^2_{\mu,X}$ in Definition \ref{def:lrd} is the energy of measure $\mu$ with symmetric kernel $K(u,v)=\int\limits_T |\cov_X(t,u,v)|\,dt$, cf. \cite[p. 77 ff.]{Landk}.

\paragraph{Self--similar random fields.}
We conclude this section with the formulation of the long range dependence in a special case of
self--similarity.

Let $X=\{X_t,t\in \R^d_+\}$ be a real valued multi--self--similar random field. By definition, it is stochastically continuous and there exist numbers $H_1,\ldots,H_d>0$ such that for a diagonal matrix $A=\mbox{diag}(a_1,\ldots,a_d)$ with $a_1,\ldots,a_d>0$ it holds
$$
\{   X_{At}, \; t\in\R^d_+\} \eqd \{ a_1^{H_1}\ldots a_d^{H_d}  X_t, \; t\in\R^d_+\} \,.
$$
Introduce the notation ${\bf 1}=(1,\ldots, 1)\in\R^d_+$ and $e^s=(e^{s_1}, \ldots,  e^{s_d} )$ for $s=(s_1,\ldots, s_d)\in\R^d$.   By \ \cite[Proposition 6]{DavPau18}, the field
$$
Y=\{  Y_s= e^{-\sum_{j=1}^d s_j H_j } X_{e^s}, \; s\in\R^d\}
$$
is stationary.
Using  Definition \ref{def:lrd} for $Y$ together with the substitution $t_i=e^{s_i}$, $i=1,\ldots, d$, we say that $X$ is s.r.d. if for any finite measure $\mu$ on $\R$ it holds
$$
\int\limits_{\R^d_+}\int\limits_{\R^2} \left| {\cov}\left(\ind \Big(X_{\bf 1}>u\Big),\ind\Big(X_{\bf t}>v\cdot \prod_{j=1}^d  t_j^{H_j}\Big) \right)    \right|\,  \frac{\mu(du)\, \mu(dv)\, dt}{\prod_{j=1}^d  t_j } <+\infty,
$$
where $dt=dt_1\ldots dt_d$ means integration with respect to Lebesgue measure in $\R^d_+$.
On the contrary, $X$ is l.r.d. if the above integral is infinite for some finite measure $\mu$ on $\R$.

\subsection{Checking the short or long range dependence} \label{subsec:CheckLRD}

Denote by $P_{\mu}(\cdot)=\mu(\cdot)/\mu(\R)$ the probability measure associated with the finite measure $\mu$ on $\R$. Let $U,V$ be two independent random variables with distribution $P_\mu$. Then the variance $\sigma^2_{\mu,X}$ from Definition \ref{def:lrd}  {becomes}
\begin{equation}\label{eq:VarEquiv}
\frac{\sigma^2_{\mu,X}}{\mu^2(\R)}= \int\limits_T \E |{\cov}_X(t,U,V)|\,dt= \int\limits_T \E |F_{X_0,X_t}(U,V)-F_{X_0}(U)F_{X_t}(V)
|\,dt.
\end{equation}
This relation is useful to check the s.r.d. of $X$ by showing the finiteness of  $\sigma^2_{\mu,X}$  for any i.i.d. random variables $U$ and $V$.
Definition \ref{def:lrd} is equivalent to the following lemma.
\begin{lemma} \label{lemm:equivDefLRD}
A stationary real--valued random field $X$  with marginal distribution function  $F_X$ is s.r.d. in the sense of Definition \ref{def:lrd} if
$$
\int\limits_T\int\limits_{{(Im \, F_X)^2}}| C_{0,t}(x,y)-xy |\, P_0(dx)\,  P_0(dy)\,dt<+\infty
$$
for any probability measure $P_0$ on ${Im \, F_X}$ where $C_{0,t}$ is the copula of the bivariate distribution of $(X_0, X_t)$, $t\in T$, and $Im \, F_X=F_X(\bar{\R})\subseteq [0,1]$ is the range of $F_X$ on $\bar{\R}=\R\cup\{+\infty\}\cup\{-\infty\}$.
$X$ is l.r.d. in the sense of Definition \ref{def:lrd} if there exists a probability measure $P_0$ on $Im \, F_X$ such that
the above integral is infinite.
\end{lemma}
\begin{proof}
By relation \eqref{eq:VarEquiv} and Sklar's theorem (cf. e.g. \cite[Theorem 2.2.1]{DurSemp16})
we have
for any finite measure $\mu$ on $\R$
$$
\sigma^2_{\mu,X}={\mu^2(\R)} \int\limits_T\int\limits_{\R^2}|C_{0,t}(F_X(u),F_X(v))-F_X(u)F_X(v)|\,    P_{\mu}(du)\, P_{\mu}(dv)  \,dt.
$$
The choice of $C_{0,t}$ is unique on $Im \, F_X$, cf. \cite[Lemma 2.2.9]{DurSemp16}.
Applying the substitution $x=F_X(u)$, $y=F_X(v)$  we get that
$$
\sigma^2_{\mu,X}={\mu^2(\R)} \int\limits_T\int\limits_{(Im \, F_X)^2}| C_{0,t}(x,y)-xy |\, P_0(dx)\, P_0(dy)\,dt,
$$
where  the probability measure $ P_0$ has a cumulative distribution function $\mu\left(  (-\infty, F_X^-(x)) \right)$, $x\in[0,1]$, and $F_X^-$ is the generalized inverse for $F_X$.
\end{proof}
Lemma \ref{lemm:equivDefLRD} implies that the new definition of memory is marginal--free, i.e., independent of the distribution of marginals $F_X$, if   $Im \, F_X=[0,1]$,
which is the case for absolutely continuous $F_X$. It essentially involves only the bivariate dependence structure encoded in the copula $C_{0,t}$.

If condition \eqref{eq:property} holds then application of the Fubini--Tonelli theorem leads to
\begin{equation*}\label{eq:sigma^2modified}
\sigma^2_{\mu,X}=\mu^2(\R) \int_T \cov \left( F_\mu (X_0),  F_\mu (X_t) \right) \, dt,
\end{equation*}
where $F_\mu (x)=P_\mu((-\infty, x))$ is the (left--side continuous) distribution function of probability measure $P_\mu$. In this case, the s.r.d. condition $\sigma^2_{\mu,X}<+\infty$ reads as a classical covariance summability property of the subordinated random field $Y_t=F_\mu (X_t)$, $t\in T$.

By stationarity of $X$, it holds $\cov_X(t,u,v)=\cov_X(-t,u,v)$ for any $t,-t\in T$, $u,v\in\R$. {Hence, in order to show l.r.d.} for $T=\R$ it is enough to check that
$$
\int\limits_0^\infty |\cov_X(t,u_0,u_0)|\,dt=+\infty $$
for some $u_0\in\R$.  For $T=\Z$ it is sufficient to consider
$
\sum_{t=1}^\infty |\cov_X(t,u_0,u_0)| =+\infty. $

\subsubsection{The short-range dependence for mixing random fields}
Let  $\mathcal{U}, \mathcal{V}$ be two sub-$\sigma-$algebras of $\mathcal{F}$. Introduce the \emph{ $z$--mixing coefficient} $z(\mathcal{U},\mathcal{V})$  (where $z\in\{\alpha,\beta,\phi, \psi, \rho \}$) as in \cite[p.3]{Doukhan94}. For instance, it is given for $z=\alpha$  by
$$
\alpha(\mathcal{U},\mathcal{V}) = \sup\left\{  \left| P(U \cap V)-P(U)P(V) \right|: \;U\in \mathcal{U}, \;V\in \mathcal{V} \right\}.
$$
Let $X=\{X_t, t\in T\}$ be a random field. Let $X_C=\{X_t, t\in C\}$, $C\subset T$, and $\sigma_{X_\mathcal{C}}$ be the $\sigma-$algebra generated by $X_C$. If $|C|$ is the cardinality of  a finite set  $C$ then the $z$-mixing coefficient of $X$ is given by
		\[ z_X(k, u, v) = \sup\{z(\sigma_{X_A},\sigma_{X_B}): \; d(A,B) \geq k,\; |A| \leq u, |B| \leq v\},\]	 	
		where $ u,v \in \N $ and $d(A,B)$ is the Hausdorff distance between finite subsets $A$ and $B$ generated by the metric on $\Rd$.
	The  interrelations between different mixing coefficients $z_X$, $z\in\{\alpha,\beta,\phi, \psi, \rho \}$  are given e.g. in \cite[p.4, Proposition 1]{Doukhan94}.

We state the result that links mixing properties and {the} short-range dependence.
The field $X$ may be non--Gaussian and have infinite variance.

\begin{theorem}\label{theo:Mixing}
		Let $X = \{X_t, t\in T\}$ be a stationary random field with  $z-$mixing rate  satisfying
		$\int_T z_X(\|t\|,1,1) \, d t < +\infty$
		where $z\in\{\alpha,\beta,\phi, \psi, \rho \}$.
		Then X is s.r.d. in the sense of Definition \ref{def:lrd}  with
		\[\int_T\int_{\mathbb{R}^2} |\cov_X(t, u, v)| \, {\mu}(d u) \, {\mu}(d v) \, d t \leq 8\int_T z_X(\|t\|,1,1) \, d t \cdot \mu^2(\R) < +\infty.\]
	\end{theorem}

\begin{proof}
Without loss of generality, we prove the result for $\alpha$-mixing $X$.
Introduce random variables $\xi(u)=\ind(X_0>u)$, $\eta(v)=\ind(X_t>v)$, where $t \in T$, $u,v \in \mathbb{R}.$
		Then, by the covariance inequality in \cite[p. 9, Theorem 3]{Doukhan94} connecting the covariance of random variables with their mixing rates we have
		\begin{align*}	&\int_T\int_{\mathbb{R}^2} |\cov_X(t, u, v)| \mu(\ud u) \mu( \ud v) \ud t = \int_T\int_{\R^2} |\cov(\xi(u), \eta(v))| \mu(\ud u) \mu( \ud v) \ud t \\ &\le
		8 \int_T \alpha (\sigma_{X_0},\sigma_{X_t})\ud t \int_{\R^2 } \| \xi(u) \|_{\infty}  \| \eta(v)\|_{\infty} \mu (du) \mu (dv)  \\ &\leq
	8	\int_T   \alpha_X(\|t\|,1,1)\ud t \cdot \mu^2(\R) < +\infty,
		\end{align*}
		where $ \| Y \|_{\infty} = \mbox{Ess-sup} (Y)$.
\end{proof}
To illustrate the  above theorem, we let $Y=\{  Y_t, \, t\in \N \}$ to be a stationary a.s. non-negative $\psi-$mixing random sequence with univariate cumulative distribution function $F_Y$ and $\int_{\Rd} \psi_Y(\|t\|,1,1) \, d t < +\infty.$ Examples of $\psi$--mixing random sequences can be found e.g. in \cite[Example 4, p. 19]{Doukhan94} (see also references therein), \cite[Theorem 2.2]{Hein96}, \cite[Proof of Claim 2.5]{Rap17}, \cite{Bradley83}, \cite[p. 54-55]{Samur}. Let $F^{-1}_Z$ be  the quantile function of a random variable $Z$ with $\E Z^2=+\infty$.
Set $G(x)=F^{-1}_Z(F_Y(x))$, $x\ge 0$, then $X_t=G(Y_t)$, $t\in \N$ is  $\psi$--mixing as well. Moreover, it is s.r.d. by the last theorem and has infinite variance because of $X_0\eqd Z$.
\begin{remark}
For a Gaussian $\phi$--mixing random field $X$, the statement of Theorem \ref{theo:Mixing} is trivial, since such $X$ is $m$--dependent \cite[Theorem 17.3.2]{IbrLinnik71}, and the integral $\int\limits_0^\infty |\cov_X(t,u,v)|\,dt$ in Definition \ref{def:lrd} is bounded by $2m$ for any $u,v\in\R$.
\end{remark}

\subsection{Subordinated Gaussian random fields}
\label{subsec:SubGaussian}

Recall that $\varphi(x) $ is the density of the standard normal law. We use the notation $\Phi(x)$ for its c.d.f. Introduce the Hermite polynomials $H_n$ of degree $n\in \N_0$ by
$$
H_n(x)=(-1)^n  \varphi^{(n)}(x) / \varphi (x)
$$
where $\varphi^{(n)}$ is the $n$-th derivative of $\varphi$.
Clearly, it holds
\[H_0(x)=1,\quad  H_1(x)=x,\quad\ H_2(x)=x^2-1,\quad H_3(x)=x^3-3x, \quad \ldots\]
For even orders $n$, Hermite polynomials are even functions, whereas for odd $n$ they are odd functions.
It is well known that Hermite polynomials form an orthogonal basis in $L^2_{\varphi}(\R)$. For any function $f\in L^2_{\varphi}(\R)$ with $\langle f, 1\rangle_{\varphi} = 0$  let
$$
\rg (f)=\min\{n\in\N: \langle f, H_n\rangle_{\varphi} \neq 0 \}
$$
be the Hermite rank of $f$. Furthermore, the Hermite rank can also be defined for functions $f\not\in L^2_{\varphi}(\R)$, as long as $\langle  |f|^{1+\theta}, \varphi\rangle<\infty$ for some $\theta\in (0,1)$; see \cite{sly:heyde:2008} or \cite[Section 4.3.5]{beran:kulik:2013}.

Let $Y = \{Y_t, t\in T\}$ be a stationary centered Gaussian real-valued random field with $\var \, Y_t=1$ and $C_Y(t)=\cov(Y_0,Y_t),$  $t\in T$. The subordinated Gaussian random field $X$ is defined by
$X_t=G(Y_t), \; t\in T,$ where $G: \mathbb{R}\rightarrow Im(G)\subseteq\mathbb{R}$ is a measurable function.

Assume first that $X$ is square integrable.
The following lemma is proven in    \cite[Lemma 10.2]{Rozanov67}:   
\begin{lemma}\label{lem:CovGauss}
Let $Z_1, Z_2$ be  standard normal random variables with $\rho = cov(Z_1,Z_2)$, and let $F$, $G$ be functions satisfying  $\E F^2(Z_1), \E G^2(Z_1) < +\infty$. Then
    \[\cov(F(Z_1), G(Z_2)) = \sum_{k=1}^{\infty}\frac{\langle F, H_k\rangle_{\varphi}
    \langle G, H_k\rangle_{\varphi} }{k!}\rho^k . \]
\end{lemma}
{Let $C_X(t)=\cov(X_0,X_t),$ $t\in T$.}
Assuming $C_Y(t)\ge 0$ for all $t\in T$ and applying this lemma to our subordinated process $X=G(Y)$ we get that it is s.r.d. in the sense of Definition \ref{def:lrd}  if
 \begin{equation}\label{eq:CondSRDFVar}
    \int_T|C_X(t)| \, d t  =   \sum_{k=1}^{\infty}\frac{\langle G, H_k\rangle_{\varphi}^2}{k!}\int_T C_Y^k(t)\, d t <+\infty.
\end{equation}
We shall see that an analogous result holds also if $X$ has no finite second moment. Introduce the condition
\begin{enumerate}
\item[{\bf ($\rho$)}]   $ |C_Y(t)|<1$ for all $t\neq 0$ if  $T$ is countable and for $\nu_d$--almost every $t\in T$ if $T$ is uncountable.
\end{enumerate}
The following result gives the conditions for s.r.d of a subordinated Gaussian random field, {without a moment assumption}. Its proof is given in Appendix.
\begin{theorem}\label{thm:subGaussSRD}
Let $Y$  be a Gaussian random field introduced above.  Let $X$ be  a subordinated Gaussian random field defined by $X_t=G(Y_t),$ $t\in T,$ where $G$ is a right-continuous strictly monotone (increasing or decreasing) function.
Assume that the condition {\bf ($\rho$)} holds. Let
\begin{equation}\label{eq:b_k}
b_k(\mu)=\Big(\int_{Im(G)}H_{k}(G^-(u))\varphi(G^-(u)) \, \mu(d u)\Big)^2
\end{equation}
where $G^-$ is the generalized inverse of $G$ if $G$ is increasing or of $-G$ if $G$ is decreasing. Then $X$ is s.r.d. in the sense of Definition \ref{def:lrd} if and only if
\begin{equation}\label{eq:condSRD_general}
\sum_{k=1}^{\infty}\frac{b_{k-1}(\mu)}{k!} \int_T |C_Y(t)| C_Y^{k-1}(t) \, d t <+\infty
\end{equation}
for any finite measure $\mu$ on $\R$.
\end{theorem}

\begin{corollary}\label{cor:subGauss}
Assume that the conditions of Theorem \ref{thm:subGaussSRD} hold.
\begin{enumerate}
\item[{\rm (i)}] Let $\mu(dx)=f(x) \, dx$ for $f\in L^1(\R)$ such that $f(x)\ge 0$ for all $x\in\R$.  If $G\in C^1(\R)$ and Im$(G)=\mathbb{R}$ then $b_k (\mu)= \langle G^\prime f(G), H_k\rangle_{\varphi}^2,$ $k\in \N$.  In this case, all coefficients \emph{$b_k(\mu)$ are finite} if for some $\theta \in (0,1)$  it holds $E[|G^\prime (Y_0)f(G(Y_0))|^{1+\theta}] < +\infty.$ If $G^\prime f(G)$ is an even function then $b_k(\mu)=0$ for all natural odd $k$.
\item[{\rm (ii)}] If $X_t=G(|Y_t|),$ $t\in T,$ then the s.r.d. condition  \eqref{eq:condSRD_general}  simplifies to
\begin{equation}\label{eq:condSRD_even}
\sum_{k=1}^{\infty}\frac{b_{2k-1}(\mu)}{(2k)!} \int_T  C_Y^{2k}(t) \, d t <+\infty.
\end{equation}
\end{enumerate}
\end{corollary}
\begin{remark}\label{rem:subGauss_LRD}
Based on Theorem \ref{thm:subGaussSRD} and Corollary \ref{cor:subGauss},  the l.r.d.
in the sense of Definition \ref{def:lrd}
 can also be formulated as follows:
\begin{enumerate}
\item[{\rm (i)}]$X=G(Y)$ is l.r.d. if $\exists u_0\in\R:$ $b_k(\delta_{\{u_0\}})<+\infty$ for all $k$ and the series \eqref{eq:condSRD_general}   diverges to $+\infty$.
\item[{\rm (ii)}] If the initial process $Y$ is s.r.d. then all powers of $C_Y$ are integrable on $T$ and the long memory of $X=G(|Y|)$ can only come from function $G$. This can happen e.g. if its \ coefficients $b_k(\mu)$ decrease to zero slowly enough.
Conversely, assume that $Y$ is l.r.d., $ 0 < b_{2k-1}(\mu)< +\infty$ for all $k \in \mathbb{N}$ and some finite measure $\mu$. If  there exists $ k\in \mathbb{N}$ s.t.  $\int_T C_Y^{2k}(t) \, d t = +\infty $
then $X$ is l.r.d.
\end{enumerate}
\end{remark}
Let us illustrate the last point of Remark \ref{rem:subGauss_LRD} by an example.

\begin{example}\label{ex:e_x2}

Let $G(x)=e^{x^2/(2\alpha)}$, $\alpha>0$,  $T=\R^d$. Then it is easy to see that
$$
P (|X_0|>x)=L(x) x^{-\alpha},
$$
where $L(x)=\sqrt{ 2/(\pi \log x)  } $.   For $\alpha\in(1,2]$, it holds
$\E \, X_0<\infty$, $\E \, X_0^2=+\infty$.

To compute $b_{2k-1}(\mu)$, we notice that
$$
\sqrt{b_{2k-1}(\mu)}=\frac{1}{\sqrt{2\pi}}\int\limits_1^\infty u^{-\alpha} H_{2k-1} (\sqrt{2\alpha \log u}) \, \mu(du), \quad k\in \N.
$$
Using the upper bound
$
 | H_{2k-1}(x) | \le x   e^{x^2/4} (2k-1)!! /4$, $x\ge 0$
from \cite[p. 787]{AbrSte72} one can show that
\begin{equation*}
\begin{split}
b_{2k-1}(\mu) &\le  \frac{  \alpha}{16\pi} [(2k-1)!!]^2 \left( \int_1^\infty u^{-\alpha/2} \sqrt{\log u} \, \mu(du) \right)^2 \\
&\le  \frac{  \alpha }{  4\pi  } \mu^2\big([1,+\infty)\big) [ (2k-1)!!]^2 <+\infty \\
\end{split}
\end{equation*}
for all $ k\in\N.$

{We note that the use of the finite measure $\mu$ is crucial here, since e.g. in case of the Lebesgue measure the integral $\int_1^\infty u^{-\alpha/2} \sqrt{\log u} \, \mu(du)$ is infinite for $\alpha\leq2$.}

Now by Stirling's formula \cite[Theorem 1.4.2]{AndrewsAskeyRoy}, we get
\begin{equation}\label{eq:asympt}
\frac{[(2k-1)!!]^2 }{  (2k)! } \sim \frac{c_3}{\sqrt{k}}, \quad k\to +\infty
\end{equation}
for $c_3>0$, so
\begin{equation}\label{eq:bk_2k}
\frac{b_{2k-1}(\mu)}{(2k)!}= O\left(      \frac{1}{\sqrt{k} } \right) , \quad k\to +\infty.
\end{equation}
Assume that  $C_Y(t)\sim \|t\|^{-\eta}$ as $\|t\|\to +\infty$, $\eta>0$.
Then $X=e^{Y^2/(2\alpha)}$, $\alpha >0$, is
\begin{itemize}
\item l.r.d. if $\eta\in (0, d/2]$ since then
{\begin{equation*}
\sum_{k=1}^{\infty}\frac{b_{2k-1}(\mu)}{(2k)!} \int_T  C_Y^{2k}(t) \, d t =+\infty.
\end{equation*}
}
\item s.r.d. if $\eta> d/2$  since then  we  have
$$\int_{\R^d} C_Y^{2k}(t)\, dt =O(k^{-1}) \quad \mbox{as } \; k\to +\infty,$$
and the series \eqref{eq:condSRD_even} behaves as
$
\sum\limits_{k=1}^\infty \frac{1}{k^{3/2} }<+\infty.
$
\end{itemize}
Here the source of long memory of $X$ is the l.r.d. field $Y$.
If $\alpha>2$  the variance of $X_0$ is finite, and our results agree with the definition in \eqref{lrdfv} by relation
\eqref{eq:CondSRDFVar} if we notice that $\rg (G)=2$. {However, the main point of this example is that we have the same transition from short to long memory (that is $\eta=d/2$) for both finite- and infinite variance fields.}

Note that for $\eta\in (d/2, d)$ the Gaussian field $Y$ is l.r.d. but the subordinated field $X=e^{Y^2/(2\alpha)}$ is s.r.d. {This agrees with the classical theory in case of finite variance, but is novel in case of infinite variance.}
\end{example}

\subsection{Stochastic volatility models}
\label{subsec:svm}

We present a way of constructing random fields with long memory by introducing a random  volatility $G(Y_t)$ (being a deterministic function of a random scaling field $Y=\{ Y_t, t\in T \}$) of a random field $Z=\{ Z_t, t\in T \}$. We assume that $Y$ and $Z$ are independent. An overview of random volatility models and their applications in finance can be found in   e.g.  \cite{Shep05} and \cite[Part II]{AndDavKreMik09}.  For each $t\in T$, $X_t=G(Y_t)Z_t$ is a scale mixture of $G(Y_t)$ and $Z_t$, see  \cite[Chapter VI, p. 345]{SteuHarn04}. Let $\bar{F}_Z=1-F_Z$ be the marginal tail distribution function of $Z_t$ for stationary $Z$.


For a finite measure $\mu$,  introduce the functional
$$
{D_{\mu}}\left(G(Y),Z_0\right)=\int\limits_T\int\limits_{\R^2}\cov\left( \bar{F}_Z\big(u/G(Y_0)\big),
\bar{F}_Z\big(v/G(Y_t)\big) \right)\mu(\ud u) \mu( \ud v)\,dt.
$$
The next lemma follows trivially from relation \eqref{eq:condcov},  independence of $Y$ and $Z$ and Tonelli theorem.

\begin{lemma}\label{thm:lrd_randvol}
Let a random field  $X=\{ X_t, t\in T \}$ be given by $X_t=G(Y_t)Z_t$ where
 $Y=\{ Y_t, t\in T \}$ and  $Z=\{ Z_t, t\in T \}$ are independent stationary random fields, $Z$ has property \eqref{eq:property}, $G: \R\to\R_{\pm}$ and $P\big(G(Y_t)=0\big)=0$ for all $t\in T$. Then
\begin{multline}\label{eq:covRandVol}
\int\limits_T \int\limits_{\R^2}
\cov_X(t,u,v)
\, \mu(du)\,\mu(dv)\,dt= D_{\mu}\left(G(Y),Z_0\right) \\ +  \int\limits_T \int\limits_{\R^2}
\E \left[ \cov_Z(t,u/G(Y_0),v/G(Y_t))\right]
\, \mu(du)\,\mu(dv)\,dt.
\end{multline}
\end{lemma}
Let us  illustrate the use of Lemma \ref{thm:lrd_randvol}.
\begin{corollary}\label{cor1}
Let the random field $X$  be given by $X_t=A Z_t$, $t\in T$, $|T|=+\infty$ where $A>0$ a.s., $A$ and $Z$ are independent and $Z\in\pa$ is stationary.  Then $X$ is l.r.d.
in the sense of Definition \ref{def:lrd}
if there exists $u_0\in \R$:
$\bar{F}_Z\big(u_0/A\big)\neq const$ a.s.
\end{corollary}

The above corollary evidently holds true if e.g.
 $Z_0\sim {\mbox Exp}(\lambda)$, $A\sim \mbox{Frechet}(1)$ for any $\lambda>0$.
It also clearly applies to a subgaussian random field  $X$ where $A=\sqrt{B}$, $B\sim S_{\alpha/2} \left( \left( \cos \frac{\pi\alpha}{4}\right)^{2/\alpha},1,0  \right) $, $\alpha \in (0,2)$, and  $Z$ is a centered  stationary Gaussian random field with covariance function $C(t)\ge 0$ for all $t\in T$ and a non--degenerate tail $\bar{F}_Z$.

The following corollary describes the situation where
light-tailed $Y$ is responsible for the l.r.d. of $X$, while  $Z$ -- for heavy tails.
\begin{corollary}\label{cor:rvt0}For the random field $X=\{ X_t, t\in T\}$  given by
$X_t=Y_t Z_t$, $t\in T$, assume that random fields $Y=\{ Y_t, t\in T\}$ and $Z=\{ Z_t, t\in T\}$ are stationary and independent.
Assume that $Z_0$ has a regularly varying tail, that is,
 $P(Z_0>x)\sim L(x)/x^\alpha$ as $x\to +\infty$ for some $\alpha>0$  where the function $L$ is slowly varying at $+\infty$.
For $Y_0>0$ a.s. assume that $\E Y_0^{\delta}<\infty$ and $\E \left( Y_0^{\delta}Y_t^{\delta} \right)<\infty$ for some $\delta>\alpha$ and all $t\in T$. Let $Y, Z\in$ \pa(\na). Then $X$ is l.r.d. if  $Y^\alpha=\{Y_t^\alpha, \,t\in T\}$ is l.r.d.
\end{corollary}

Now we scale a l.r.d. (possibly heavy--tailed)  random field $Z$ by a random volatility $G(Y)$ being a subordinated Gaussian random field.
\begin{lemma}\label{lemm:DSubGauss}
Let $X_t=G(Y_t) Z_t$ be a random field as in Lemma \ref{thm:lrd_randvol}. Assume additionally that $Y$ is a centered Gaussian random field with unit variance and covariance function $\rho(t)\ge 0$ satisfying condition {\bf ($\rho$)}.
Then
\begin{equation*}
D_{\mu}\left(G(Y),Z_0\right)= \sum\limits_{k=1}^\infty \frac{ \left( \int_{\R} \langle \bar{F}_Z (u/G(\cdot)),  H_k(\cdot) \rangle_\varphi \, \mu(du)  \right)^2 }{k!}\int_T \rho^k(t) \, dt .
\end{equation*}
\end{lemma}

The following example illustrates our definition of l.r.d. in the context of a popular {\it long memory stochastic volatility model} that is used in econometrics to model log--returns of stocks, see \cite[p.70ff]{beran:kulik:2013} and references therein.


\begin{example}\label{ex:RV1}
Assume that $X=\{X_t,t\in\Z\}$ has a form $X_t=e^{Y_t^2 /4}Z_t$, where $Z_t$ is a sequence of i.i.d. random variables {with finite moment of order $2+\delta$ for some $\delta>0$}, while $Y_t$ is a centered stationary Gaussian \pa \ long memory sequence with unit variance and covariance function $C_Y$ satisfying condition ($\rho$). Both sequences $Z_t$ and $Y_t$ are assumed to be independent from each other. From Example \ref{ex:e_x2} we know that $e^{Y_0^2 /4}$ is regularly varying with index $\alpha=2$. By Breiman's lemma  the tail distribution function of $|X_0|$ is also regularly varying with index $\alpha =2$ and hence $X_0$ has infinite variance. Choose
$\mu=\delta_{\{  u_0\}} $ for some $u_0\in\R$. Lemmas \ref{thm:lrd_randvol} and \ref{lemm:DSubGauss} yield
\begin{equation}\label{eq:C_Xiid}
\sum\limits_{t=1}^{\infty} \int\limits_{\R^2} {\cov}_X(t,u,v) \mu(du)\mu(dv) = \sum\limits_{k=1}^\infty \frac{ \langle \bar{F}_Z(u_0/G),  H_k\rangle_\varphi^2 }{k!} \sum_{t=1}^{\infty} C_Y^k(t)  ,
\end{equation}
where $G(x)=e^{x^2/4}$. Since $\bar{F}_Z(u_0/G)$ is symmetric, monotone nondecreasing and bounded we get $\langle \bar{F}_Z(u_0/G),  H_{k}\rangle_\varphi=0$ for all odd $k$, and it is finite for all even $k\in\N$. Moreover, by Lemma \ref{lemm:HermRank}, 2) we have
$\rg(\bar{F}_Z(u_0/G) )=2$. It is clear then that $X$ is l.r.d.  if  $\sum_{t=1}^\infty\rho^2(t)=+\infty$. In particular, if $C_Y(t)\sim |t|^{-\eta}$ as $|t|\to\infty$, then l.r.d. occurs if $\eta\in (0,1/2]$.
{Again, similarly to Example \ref{ex:e_x2}, the point here is that we obtain long memory in case of both finite and infinite variance.}
\end{example}


\section{Limit theorems}
\label{sec:lt}

In this section, we investigate connections between Definition \ref{def:lrd} and limit theorems for random volatility and subordinated Gaussian random fields. In order to do so, we have to specify the statistic  whose limiting behaviour we consider.
We focus on the volume of the excursion sets.

In Section \ref{sec:subordinated}
we consider subordinated Gaussian random fields. In Section \ref{subsect:LTInfVar}
we show by a natural example that our definition of long memory is in agreement with the existing limiting behaviour of the volume of excursions of $X$ over some levels $u$.
On the other hand, in Section
\ref{sec:mean-infinite-variance},
we will indicate that the limiting behaviour of the empirical mean cannot be directly related to our definition. The latter is not surprising.

In Section \ref{subsec:LTInt} we consider related problems for stochastic volatility random fields.

From now on, we assume the random field $X$ to be measurable. In what follows, $L$ will indicate a slowly varying function at infinity, that can be different at each of its occurrences.

We start with the following lemma that will play a major role.
\begin{lemma}\label{lemm:HermRank}
Let $Y,$ $Z$ be independent random variables such that  $Y\sim N(0,1)$.   For any monotone right-continuous non--constant function  $G:\R\to\R_\pm$ with $\nu_1\left(\{x\in\R:  G(x)=0\} \right)=0 $, consider the functions $\widetilde{G}(y)=G(|y|)$ and
\begin{equation}\label{eq:psi}
\zeta_{G,Z,u}(y)=    \E[\ind \{G(y)Z>u\} ]-P \left( G(Y)Z>u\right), \quad y\in \R
\end{equation}
for a fixed $u> 0$ if $G\ge 0$ and $u<0$ if $G\le 0$. Then the following holds:
\begin{enumerate}
\item[{\rm (i)}] Let $G:\R\to\R_\pm$ be as above such that $\E |G(Y)|^{1+\theta}<+\infty $ for some $\theta\in(0,1]$. Then
$  \rg (G)= \rg(  \zeta_{G,1,u}) = \rg(  \zeta_{G,Z,u})=1.$
\item[{\rm (ii)}] Let $G:\R_+\to\R_\pm$ be as above such that $\E |\widetilde{G}(Y)|^{1+\theta}<+\infty $ for some $\theta\in(0,1]$, $G^-(u)\neq 0$, where  $G^{-}$ is the generalized inverse of $G$. Then
$  \rg (\widetilde{G})=\rg( {\zeta}_{\widetilde{G},1,u})= \rg( {\zeta}_{\widetilde{G},Z,u})=2.$
\end{enumerate}
\end{lemma}

\begin{remark}\label{rem:G}
\begin{enumerate}
\item[{\rm (i)}] If $Z\equiv 1$ the assertion of Lemma \ref{lemm:HermRank}(i) holds under milder assumptions on $G$ and $u$. Thus,  let $G:\R\to\R$ be a monotone right--continuous non--constant function such that $\E |G(Y)|^{1+\theta}<+\infty $ for some $\theta\in(0,1]$. Then for any $u\in\R$  $\rg ( G) =\rg( \zeta_{G,1,u})=1$.
\item[{\rm (ii)}] The assumption of nonnegative or nonpositive $G$ is essential to the statement $  \rg(  \zeta_{G,Z,u})=1$
of Lemma \ref{lemm:HermRank}(i) since for $G(y)=y$ and symmetric  $Z$ we have
$
\E[Y \ind \{  YZ >u \}] =0,
$
so the Hermite rank of $\zeta_{G,Z,u} $  is greater than 1. Similarly, one can construct examples of functions $G$ with
$\rg(  \zeta_{\widetilde{G},Z,u})>2$ for some $u\in\R$  if the assumptions of Lemma \ref{lemm:HermRank}(ii) do not hold. For instance, $G^-(u)=0$ means that $\rg(  \zeta_{\widetilde{G},Z,u})\ge 4$.
\item[{\rm (iii)}] If $G$ is nonnegative or nonpositive and $u=0$ then it is easily seen that $\zeta_{G,Z,0}\equiv 0$ and, formally speaking, its Hermite rank is infinite.
\end{enumerate}
\end{remark}

\subsection{Limit theorems for subordinated Gaussian processes}\label{sec:subordinated}
Let $X=\{  X_t, t\in\R^d \}$ where $X_t= G(Y_t)$ and $Y=\{ Y_t, t\in \R^d \}$ is a stationary isotropic l.r.d. centered Gaussian random field with covariance function
$C_Y(t)= \|t\|^{-\eta}L(\|t\|),$
$\eta\in (0,d/q)$ (cf. \cite{IvLeon89,Leonenko99,LeonOlenko14}). 
Here $\E G^2(Y_0)<+\infty$ and $q$ is the Hermite rank of $G$.  Under some technical assumptions on
the spectral density $f(\lambda)$  of $Y$ (cf. \cite[Assumption 2]{LeonOlenko14})
it holds
\begin{align}\label{eq:non-clt:subGauss}
n^{q\eta/2-d}L^{-q/2}(n)
\int_{W_n}G(Y_t)\, dt \tod R\;, \quad n\to+\infty,
\end{align}
 where
 \begin{equation}\label{eq:R}
 R=\left( \gamma(d,\eta) \right)^{q/2}\int^\prime_{\R^{dq}}  \int_W e^{i\langle \lambda_1+\ldots+\lambda_q,u\rangle} du \frac{\tilde{B}(d\lambda_1) \ldots \tilde{B} (d\lambda_q)}{\left( \| \lambda_1\| \cdot \ldots \cdot  \| \lambda_q\|   \right)^{(d-\eta)/2}},
\end{equation}
$$
\gamma(d,\eta) =\frac{\Gamma\left( (d-\eta)/2 \right)  }{2^\eta \pi^{d/2} \Gamma(\eta/2)},
$$
and $\int^\prime_{\R^{dq}}  $ is the multiple Wiener--Ito integral with respect to a complex Gaussian white noise measure $\tilde{B}$ (with structural measure being the spectral measure of $Y$, cf. \cite[Section 2.9]{IvLeon89}). It is easy to see that in case $q=1$ the distribution of $R$ is Gaussian. However, the normalization $n^{\eta/2-d} L^{-1/2}(n)$ differs from the CLT--common normalizing factor $n^{-d/2}$  which agrees with the fact that $X$ is l.r.d. in the sense of the usual definition as in \eqref{lrdfv}. For $q\ge 2$, one gets a $q$--Rosenblatt--type distribution for $R$, see \cite{VeilletteTaqqu13,LeonRuizTaqqu17} and references therein for its properties in the case $q=2$.

\subsubsection{Volume of level sets} \label{subsect:LTInfVar}

We specify the above situation to the level sets.
Assume $G:\R\to\R$ to be a monotone right--continuous function such that $\E |G(Y)|^{1+\theta}<+\infty $ with $\theta\in(0,1)$. Let the variance of $X_0$ be infinite. For any $u\in\R$ introduce the function $g_u(x)=\zeta_{G,1,u} (x)$, where  $\zeta_{G,1,u} $ is given in \eqref{eq:psi}. By Remark \ref{rem:G}(i), the Hermite ranks of $G$ and $g_u$ are equal to one. If $\eta\in (0,d)$ then
\begin{align*}
\frac{
\int_{W_n}g_u(Y_t)\, dt}{n^{d-\eta/2}L^{1/2}(n)} = \frac{  \int_{W_n} \ind \left (G(Y_t)>u\right)\, dt    - \nu_d(W_n) P\left(G(Y_0)>u\right)      }{n^{d-\eta/2}L^{1/2}(n)}
 \tod R
\end{align*}
as $n\to+\infty$ where $R$ is given in \eqref{eq:R}.
The normalization in this limit theorem is not of CLT-type $n^{-d/2}$ which should be attributed to the l.r.d. case. Let us compare this behavior with Definition \ref{def:lrd}.
As an example, we consider
$$
G(x)=\mbox {sgn} (x) \left( e^{x^2/\beta^2}-1\right), \quad x\in\R
$$
for some $\beta>\sqrt{2(1+\theta)}$.  Note that it is possible that the variance of $X=G(Y)$ is infinite. Set $\mu=\delta_{\{ 0 \}  }$. By Remark \ref{rem:subGauss_LRD}, 1) we get $b_k(\mu)=H_k^2(0)/(2\pi)<+\infty$ for any $k\ge 0$,  $b_0>0$, $b_1= 0$, etc. By the choice $ C_Y(t)= \|t\|^{-\eta}L(\|t\|),$
$\eta\in (0,d)$ we get that $\int_{\R^d} |C_Y(t)|\, dt = +\infty$, and the series \eqref{eq:condSRD_general} diverges. Then $X$ is l.r.d. in the sense of Definition \ref{def:lrd} for $\eta\in (0,d)$  which is in accordance with the above limit theorem.

\subsubsection{Empirical mean: infinite variance case}\label{sec:mean-infinite-variance}
In this section
we show that Definition \ref{def:lrd} cannot be linked the behavior of integrals or partial sums of the field $X$ if $X$ has infinite variance. For that, we use the framework of time series
$X=\{ X_t, \ t\in\Z \}$ where many more models have been widely explored, as compared to (continuous-time) random fields.


Consider (similarly as in Section \ref{subsec:SubGaussian}) a subordinated time series $X_t=G(|Y_t|)$, $t\in\Z$, where $\{Y_t, \: t\in\Z\}$ is a centered Gaussian long memory linear time series with nondecreasing covariance function $C_Y(t)=\cov(Y_0,Y_t)\sim |t|^{-\eta}L(t) $, $t\to+\infty$, $\eta\in(0,1)$, and such that
$P(|X_0|>x)\sim  x^{-\alpha}L(x)$, $\alpha\in (0,2)$. It is further assumed that $G$ has Hermite rank $q$.
By  Corollary \ref{cor:subGauss}(ii),  $X$ is short range dependent in the sense of Definition \ref{def:lrd} whenever for any finite measure $\mu$ on $\R$
\begin{equation}\label{eq:condSRD_Ex}
\sum\limits_{k=1}^{\infty}\frac{b_{2k-1}(\mu)}{(2k)!} \sum\limits_{t=1}^{\infty}  C_Y^{2k}(t) <+\infty.
\end{equation}
We note that
\begin{equation}\label{eq:zeta_rho}
 \sum\limits_{t=1}^{\infty}  C_Y^{2k}(t) \le c_0  \int\limits_{1}^{\infty}  \frac{L^{2k}(t)}{ t^{2k\eta}}\, dt
 \le   \int\limits_{1}^{\infty}  \frac{c_1 \, dt}{ t^{2k(\eta-\delta)}}, \quad k\in \N,
 \end{equation}
 where  $\delta>0$ is arbitrary and $c_0, c_1>0$ are some constants. The second inequality holds since $L(t)\le  c_2  t^\delta$  for $t\ge t_0$ where $t_0>0$ is large enough and $c_2=c_2(\delta,t_0)=(1+\delta) L(t_0)/t_0^\delta \le 1$ for large $t_0$, cf. \cite[Proposition 2.6]{Resnick07}.
 The right--hand side of \eqref{eq:zeta_rho} is finite and equal to $O(1/k)$ whenever  $\eta\in(1/2,1)$ since $\delta>0$ can be chosen arbitrarily small. 
 The series in \eqref{eq:zeta_rho} diverges if $\eta\in (0,1/2)$. If $\eta=1/2$ the summability of the series in \eqref{eq:zeta_rho} depends on the particular form of the slowly varying function $L$ and will not be discussed here.

Thus, for $\eta\in(1/2,1)$ $X$ is s.r.d. whenever
\begin{equation}\label{eq:cond_b_k}
\sum\limits_{k=1}^{\infty}\frac{b_{2k-1}(\mu)}{(2k)! k}   <+\infty
 \end{equation}
 for any finite measure $\mu$.

Now we have to consider a special example of function $G$ in order to get more explicit results for the s.r.d. case. As in Example \ref{ex:e_x2},  set $G(x)=e^{x^2/(2\alpha)}$, $\alpha\in(0,2]$.  
By relation \eqref{eq:bk_2k}, condition \eqref{eq:cond_b_k} is satisfied for $\eta\in(1/2,1)$, hence $X$ is s.r.d. in the sense of  Definition \ref{def:lrd}
if $\eta\in(1/2,1)$  and l.r.d. if  $\eta\in (0,1/2)$.

Let us compare this result with the limiting behaviour of the partial sums $S_n=\sum_{t=1}^n (X_t-\E[X_t])$ as given  in \cite{sly:heyde:2008} and \cite[Section 4.3.5]{beran:kulik:2013}, 
cf. Table 1.  
There, some discrepancies are seen, that is Definition  \ref{def:lrd} does not  agree with the asymptotic behaviour of $S_n$.
\begin{center}
\begin{table}\label{tab:alpha12}
\begin{tabular}{|c|c|}
\hline
Parameter range & Limit of normalized sums $S_n$ \\ \hline $1-{1}/{\alpha}<\eta<1$  & $\alpha$--stable \\  \hline  $0< \eta<1-{1}/{\alpha}$ & Rosenblatt \\
\hline
\end{tabular}
\caption{Short or long memory of $X_t=e^{Y_t^2/(2\alpha)}$ in the infinite variance case $\alpha\in(1,2)$ in dependence of the long memory parameter $\eta$ of $Y$ according to
paper  \cite{sly:heyde:2008}. }
\end{table}
\end{center}

\subsection{Limit theorems for the integrals of  functionals of l.r.d. random volatility fields}\label{subsec:LTInt}
In this section we will justify that our definition of l.r.d. is in agreement with limit theorems for volumes of level sets for random volatility models. Unlike as in the subordinated Gaussian case {where the limiting results are known}, a general asymptotic theory has to be developed.

Let $X$ be a random volatility field of the form $X_t=G(Y_t)Z_t$, $t\in \Z^d$, where
\begin{itemize}
\item $\{G(Y_t),t\in \R^d\}$ is a subordinated Gaussian measurable random field, which is sampled at points $t\in\Z^d$,
\item $\{Z_t,t\in \Z^d\}$ is a white noise,
\item the random fields $Y$ and $Z$ are independent.
\end{itemize}
Our goal is to prove limit theorems for
$
\sum_{t\in W_n}g(X_t)
$
as $n\to\infty$,
where $W_n=[-n,n]^d \cap \Z^d$ and $g$ is a real valued Borel--measurable function such that
\begin{align}\label{eq:clt-cond-4}
\E [g(X_0)]=0,\quad \E [g^2(X_0)]>0\, .
\end{align}

Introduce the function
$$\xi(y)= \E [g(G(y) Z_0)] \;.$$
It follows from \eqref{eq:clt-cond-4}  that for $\nu_1$--almost every $y\in \R$
\begin{align}\label{eq:clt-cond-1}\xi(y)<\infty\;.
\end{align}
By \eqref{eq:clt-cond-4} we also have $\E[\xi(Y_0)]=0$.
Let
$$
J(m)=\langle \xi, H_m\rangle_\varphi = \E[H_m(Y_0) \, g(G(Y_0)Z_0)]\;
$$
be the $m$th Hermite coefficient of $\xi$.
We recall that a sufficient condition for the finiteness of $J(m)$ is
\begin{align}\label{eq:clt-cond-3}
\E [|g(X_0) |^{1+\theta}]=  \E [|\xi(Y_0) |^{1+\theta}]
=\E\left[\left| \E[g(G(Y_0)Z_0)\mid {\cal Y}]\right |^{1+\theta}\right]
<\infty
\end{align}
for some $\theta\in(0,1]$, {where ${\cal Y}$ is a sigma-field generated by the entire sequence $Y$}. 
Let $\rg (\xi)= q$.
Furthermore, set
$$
m(y,Z_t)=g(G(y)Z_t) - \E[g(G(y)Z_t)]=g(G(y)Z_t) - \xi(y)\;,
$$
which is almost everywhere finite by \eqref{eq:clt-cond-1},
and
$
\chi(y)=\E[m^2(y,Z_0)]\;.
$
We also assume
\begin{align}\label{eq:clt-cond-2}
\E[\chi^3(Y_0)]<\infty\;.
\end{align}
Note that under \eqref{eq:clt-cond-2}, using Lyapunov inequality on a space of finite measure and the stationarity of $Y_t$, we have for any finite subset $I\subset \Z^d$ that
\begin{align*}
\E\left[\left(\sum_{t\in I}\chi(Y_t)\right)^3\right]<\infty\;.
\end{align*}
The following result shows that the limiting behaviour is primarily determined by the function $\xi$, with $\xi\equiv 0$ being the boundary case.
\begin{theorem}\label{thm:LTrvf}
Assume that random field $X_t=G(Y_t)Z_t$, $t\in\Z^d$, is  as  above,  where additionally
\begin{itemize}
\item $Y$ is a homogeneous isotropic centered Gaussian random field with the covariance function
$
C_Y(t)=\E [Y_0Y_t]= \|t\|^{-\eta}L(\|t\|)
$,
$\eta\in (0,d/q)$ and $L$ is slowly varying at infinity,
\item $Y$ has a spectral density $f(\lambda)$ which is continuous for all $\lambda\neq 0$ and decreasing in a neighborhood of $0$.
\end{itemize}
Assume that \eqref{eq:clt-cond-4}, \eqref{eq:clt-cond-3} with $\theta=1$, \eqref{eq:clt-cond-2} hold.
\begin{enumerate}
\item If $\xi(y)\equiv 0$ then
\begin{align}\label{eq:clt}
n^{-d/2}
\sum_{t\in W_n}g(X_t) \tod {\cal N}(0,\sigma^2)\;, \quad n\to+\infty,
\end{align}
where $\sigma^2=\E[g^2(X_0)] 2^d>0$.
\item If
$\xi(y)\not\equiv 0$ then
\begin{align}\label{eq:non-clt}
n^{q\eta/2-d}L^{-q/2}(n)
\sum_{t\in W_n} g(X_t) \tod R\;, \quad n\to+\infty,
\end{align}
 where the random variable $R$ is given in \eqref{eq:R} with $W=[-1,1]^d$.
\end{enumerate}
\end{theorem}


\begin{example}
{\rm
Assume that $g(y)=y$, $\E[G^2(Y_0)]<\infty$ and $\E[Z_0]=0$.
Then $\xi(y)=G(y) \E[Z_0]= 0$ and \eqref{eq:clt} always holds. In this case, there is no contribution from the long memory of the random field $Y_t$.
}
\end{example}

\begin{example}
{\rm
Assume that $g(y)=y-\E[G(Y_0)Z_0]$, $\E[Z_0]\not=0$. Then $\xi(y)=\E[Z_0]\left\{ G(y)-   \E[G(Y_0)]\right\}$.
Condition \eqref{eq:clt-cond-2} is satisfied if  $\E[|Z_0|^3]<+\infty$, $\E[G^4(Y_0)]<+\infty$.
In this case  $\xi(y)\not\equiv 0$, and \eqref{eq:non-clt} always holds.
}
\end{example}

\begin{example}\label{eq:LTrvm}{\rm
Assume that $g(y)=g_u(y)=\ind \{y>u\}-P\left( G(Y_0)Z_0>u\right)$ where $G$ is nonnegative or nonpositive $\nu_1$--a.e. Then
$$
\xi(y)=    \E[\ind \{G(y)Z_0>u\} ]-P\left( G(Y_0)Z_0>u\right)\not\equiv 0
$$
if $u\neq 0$, so case \eqref{eq:non-clt} applies. If $u=0$ then $\xi(y)\equiv 0$ (compare Remark \ref{rem:G}(iii)), so case \eqref{eq:clt} holds true.}
\end{example}

\begin{example}{\rm
Let the random volatility field $X_t=G(|Y_t|)Z_t$, $t\in\Z^d$ be as in Lemma  \ref{lemm:DSubGauss} where $\{  Z_t\}$ is a heavy--tailed white noise,  $\E Z_0^2=+\infty$. Let $Y$ satisfy the assumptions of Theorem \ref{thm:LTrvf}.
Choose $G(x)\ge 0$ as in Lemma \ref{lemm:HermRank}(ii),  and $C_Y(t)\sim \|t\|^{-\eta}$ as $ \|t\|\to+\infty$ be nonnegative.
Similarly to Example \ref{ex:RV1}, an analogue of relation \eqref{eq:C_Xiid} holds true: for $\mu=\delta_{\{u_0\}}$, $u_0>0$ we have
\begin{equation*}
\sum\limits_{t\in \Z^d, \, t\neq 0}\int\limits_{\R^2} {\cov}_X(t,u,v) \mu(du)\mu(dv) = \sum\limits_{k=1}^\infty \frac{ \langle \bar{F}_Z(u_0/\widetilde{G}),  H_k\rangle_\varphi^2 }{k!} \sum\limits_{t\in \Z^d, \, t\neq 0}\ C_Y^k(t) ,
\end{equation*}
where $\widetilde{G}(y)=G(|y|)$, $y\in\R$.
Since $\rg(\bar{F}_Z(u_0/\widetilde{G}) )=2$, $X$ is l.r.d. in the sense of Definition \ref{def:lrd} if $\sum_{t\in \Z^d, \, t\neq 0} C_Y^{2}(t) =+\infty$, that is, if $\eta<d/2$.

Consider function $\xi $ from Example  \ref{eq:LTrvm} with $u=u_0>0$ and $\widetilde{G}$ instead of $G$. By Lemma \ref{lemm:HermRank}, 2)
 $\rg(\xi )=2$.
By Theorem \ref{thm:LTrvf} and Example \ref{eq:LTrvm}, the asymptotic behavior of the cardinality of the level sets of $X$ at niveau $u_0$  is of l.r.d.-type if $\eta\in(0,d/2)$ which is in agreement with our definition.}
\end{example}
\begin{remark}
We would like to connect the assumption $\xi\equiv 0$ to our definition. Let $g,h$ be functions such that $\E[g(X_0)]=\E[h(X_0)]=0$.
If $\E [g(X_0)h(X_t)]<\infty$ for all $t$, and $\E[g(G(y)Z_0)]=\E[h(G(y)Z_0)]=0$ for all $y$, then for $t\not=0$
\begin{align}\label{eq:sv}
\cov(g(X_0),h(X_t))=\int\int\E [g(G(y_0)Z_0)] \E[h(G(y_t)Z_t)] P_{Y_0,Y_t}(dy_0,dy_t)=0 \;,
\end{align}
where $P_{Y_0,Y_t}$ is the joint law of $(Y_0,Y_t)$. In particular, take
$$
g(x)=g_u(x)=\ind(x>u)-P(X_0>u)\;, \ \ h(x)=h_v(x)=\ind(x>v)-P(X>v)\;.
$$
Then
 $$\sigma_{\mu,X}^2=\sum\limits_{t\in\Z^d, \, t\neq 0} \int_{\R^2} |\cov(g_u(X_0),h_v(X_t))|\, du dv=0,$$ 
 and the random field $X$ is s.r.d. according to Definition \ref{def:lrd} in case $\xi\equiv 0$.
\end{remark}

\section{Summary and outlook}
We proposed a new definition of long memory for stationary random fields $X$ indexed by any set $T\subset \R^d$ which works also for heavy tailed $X$. We showed that this definition fits well the asymptotic behavior of the volume of the excursion set of $X$ at a level $u\in\R$ in a unboundedly growing observation window $W_n$. This connection to non--central limit theorems was proven for a class of random volatility fields with a subordinated l.r.d. Gaussian volatility.

\section{Appendix: Proofs}\label{sec:Appendix}

\begin{proof}[Proof of Theorem \ref{thm:subGaussSRD}]	

If $X$ is a  centered stationary unit variance Gaussian random field with
covariance function $C_Y(t)$,
\begin{equation}\label{eq:CovIndGauss}
   \cov_X(t,u,v) = \frac{1}{2\pi}\int_{0}^{C_Y(t)} \frac{1}{\sqrt{1-r^2}}\,
   \exp\left\{-\frac{u^2 -2r uv + v^2}{2\left(1-r^2\right)} \right\}\,dr,
\end{equation}
see \cite[Lemma 2]{BuSpoTim12}.

Consider representation \eqref{eq:CovIndGauss}.  Since the density $f_{(U,V)}$  of a bivariate normal distribution with zero mean, unit variances and correlation coefficient $\mp r$ equals \[\frac{1}{2\pi \sqrt{1-r^2}}\exp\left\{-\frac{x^2\pm2rxy+y^2}{2(1-r^2)}\right\} \geq 0\]  then it is easy to see that
		\begin{gather*}
			|\cov_Y(t, x, y)| = \frac{1}{2\pi}\Int_{0}^{|C_Y(t)|}\frac{1}{\sqrt{1-r^2}}\exp\left\{-\frac{x^2-2sign(C_Y(t))rxy+y^2}{2(1-r^2)}\right\} \ud r.
		\end{gather*}
		Since $G$ is strictly monotone, by properties of the generalized inverse of $G$ we have
		\begin{gather*} \Int_T\Int_{-\infty}^{+\infty}\Int_{-\infty}^{+\infty}|\cov_X(t, u, v)| \mu(\ud u) \mu (\ud v) \ud t = \\
			\Int_T\!\!\! \Int_{(Im(G))^2} \!\!\! \!\!\! |\cov_Y(t, G^-(u), G^-(v))| \mu(\ud u) \mu (\ud v)  \ud t =\\
		     \Int_T\!\!\! \Int_{(Im(G))^2} \!\!\! \Int_{0}^{|C_Y(t)|} \!\!\! \exp\left( \!\! -\frac{(G^-(u))^2-2sign(C_Y(t))rG^-(u)G^-(v)+(G^-(v))^2}{2(1-r^2)}\right) \!\! \frac{\ud r \mu(\ud u) \mu (\ud v)  \ud t}{2\pi \sqrt{1-r^2}}.
		\end{gather*}
		By \cite[Formula (21.12.5)]{Cramer46} for the density $f_{(U,V)}$  with correlation coefficient $sign(C_Y(t))r\in(-1,1)$  it holds
		\begin{equation}\label{eq:bivariateGaussDensSeries}
		f_{U,V}(x,y) = \sum_{k=0}^{\infty}\dfrac{\Phi^{(k+1)}(x)\Phi^{(k+1)}(y)}{k!}(sign(C_Y(t))r)^k ,\;\; x,y \in \mathbb{R}.
		\end{equation}
		By condition $\nu_d(\{t\in T:|C_Y(t)|=1\})=0$, the above series converges  uniformly for $r \in (-1,1)$, so  integration over $r\in [0; |C_Y(t)|]$ and summation with respect to $k$ can be interchanged. Then the above triple integral reads
				\begin{gather*} 	
			\Int_T\Int_{Im(G)^2}\Int_{0}^{|C_Y(t)|}\sum_{k=0}^{\infty}\dfrac{\Phi^{(k+1)}(G^-(u))\Phi^{(k+1)}(G^-(v))}{k!}(sign(C_Y(t))r)^k \ud r\mu(\ud u) \mu (\ud v)  \ud t\\
			=	       	\Int_T\Int_{Im(G)^2}\varphi(G^-(u))\varphi(G^-(v))\sum_{k=0}^{\infty}\dfrac{H_k(G^-(u))H_k(G^-(v))}{k!}sign(C_Y(t))^k   \\
			\times  \frac{|C_Y(t)|^{k+1}}{k+1}  \mu(\ud u) \mu (\ud v)  \ud t\\
			=\Int_T\Int_{Im(G)^2}\!\!\!  |C_Y(t)|\varphi(G^-(u))\varphi(G^-(v))\sum_{k=0}^{\infty}\dfrac{H_k(G^-(u))H_k(G^-(v))}{(k+1)k!}C_Y(t)^k  \mu(\ud u) \mu (\ud v)  \ud t.
		\end{gather*}
		Abel's uniform convergence test allows  us to interchange the sum and the integral over $Im(G)^2 $. Since  $b_k\geq 0$ we get
		\begin{gather*} 		\Int_T\sum_{k=0}^{\infty}\Int_{Im(G)^2} \varphi(G^-(u))\varphi(G^-(v))\dfrac{H_k(G^-(u))H_k(G^-(v))}{(k+1)!}|C_Y(t)|C_Y(t)^k \ud r \mu(\ud u) \mu (\ud v)  \ud t \nonumber \\
			=\Int_T\sum_{k=0}^{\infty}\frac{1}{(k+1)!}\Big(\Int_{Im(G)}\varphi(G^-(u))H_k(G^-(u)) \mu(\ud u)\Big)^2 |C_Y(t)|C_Y(t)^k \ud t\nonumber \\ =  \label{eq:LstToMemory} 	  \Int_T\sum_{k=0}^{\infty}\frac{b_k(\mu)}{(k+1)!}|C_Y(t)|C_Y(t)^k \ud t=  \sum_{k=1}^{\infty}\frac{b_{k-1}(\mu)}{k!}\Int_{T}|C_Y(t)|\rho^{k-1}(t)\ud t ,
		\end{gather*}
		where the integral over $T$ and the sum are interchangeable by Tonelli's theorem subdividing $T$ into parts $T^+ = \{t\in T: C_Y(t)\geq 0\}$ and $T^- = \{t\in T: C_Y(t)< 0\}.$
		Then  $X=G(Y)$ has short memory if
		$$
			\sum_{k=1}^{\infty}\frac{b_{k-1}(\mu)}{k!}\Int_{T}|C_Y(t)|\rho^{k-1}(t)\ud t < +\infty
		$$
		for any finite measure $\mu$ on $\R$.
\end{proof}

\begin{proof}[Proof of Corollary \ref{cor:subGauss}]
\begin{enumerate}
\item   It follows from relation  \eqref{eq:b_k} using the change of variables $u=G(x)$ and by  \cite[Lemma 4.21]{beran:kulik:2013}.
\item		
W.l.o.g. assume $G$ to be an increasing function. Since the probability density of the centered uni- and bivariate Gaussian distribution is invariant under transformation $x\longmapsto-x,y\longmapsto-y$ we get
		\begin{align*}
			&\cov_X(t, u, v) = P(|Y_0|>G^{-}(u),|Y_t|>G^{-}(v))\\
			&\phantom{=}-P(|Y_0|>G^{-}(u))P(|Y_t|>G^{-}(v))\\
			&=
			2\left( P(Y_0>G^{-}(u),Y_t>G^{-}(v))-P(Y_0>G^{-}(u))P(Y_t>G^{-}(v))\right.\\
			&\phantom{=}+ \left. P(Y_0>G^{-}(u),Y_t<-G^{-}(v))-P(Y_0>G^{-}(u))P(Y_t<-G^{-}(v)) \right).
		\end{align*}
Denote $Z=-Y_t$, $x=G^-(u),$  $y=G^-(v)$.      It holds
$$P(Y_0>x,Y_t>y)-P(Y_0>x)P(Y_t>y) = \cov(\ind(Y_0\ge x), \ind(Y_t\ge y)) , $$	
$$P(Y_0>x,Y_t<-y)-P(Y_0>x)P(Y_t<-y)= \cov(\ind(Y_0>x), \ind(Z>y)) .$$		
Since   $\cov(Y_0,Z) = -C_Y(t) $ and $xy=G^{-}(u)G^{-}(v)\ge 0$  we have by formula \eqref{eq:CovIndGauss} that
\begin{gather*}
			|\cov_X(t, u, v)|  =\frac{2}{2\pi} \left|	\Int_{0}^{C_Y(t)}\frac{1}{\sqrt{1-r^2}}\exp\left(-\frac{x^2-2rxy+y^2}{2(1-r^2)}\right) \ud r \right. \\
			+\left. \Int_{0}^{-C_Y(t)}\frac{1}{\sqrt{1-r^2}}\exp\left(-\frac{x^2-2rxy+y^2}{2(1-r^2)}\right) \ud r \right| \\
			 = \Int_{0}^{|C_Y(t)|}  \left(\exp\left(-\frac{x^2-2rxy+y^2}{2(1-r^2)}\right) - \exp\left(-\frac{x^2+2rxy+y^2}{2(1-r^2)}\right)\right)  \frac{\ud r}{\pi \sqrt{1-r^2}}.
\end{gather*}	   	
		
	Similarly to the proof of Theorem \ref{thm:subGaussSRD}, we use representation \eqref{eq:bivariateGaussDensSeries} to write
		\begin{gather*}	
		   	\Int_T\Int_{-\infty}^{+\infty}\Int_{-\infty}^{+\infty} |\cov_X(t, u, v)| \mu(\ud u) \mu( \ud v) \ud t \\=
			2\Int_T\Int_{Im(G)^2}  \sum_{k=0}^{\infty}\frac{1-(-1)^k}{(k+1)!}H_{k}(x)H_{k}(y)\varphi(x)\varphi(y)|C_Y(t)|^{k+1} \mu(\ud u) \mu( \ud v) \ud t\\=
			\Int_T \sum_{k=1}^{\infty}\frac{4}{(2k)!}\left(\Int_{Im(G)}H_{2k-1}(G^-(u))\varphi(G^-(u)) \mu(\ud u) \right)^2|C_Y(t)|^{2k} \ud t  \\=	
			4\sum_{k=1}^{\infty}\frac{b_{2k-1}(\mu)}{(2k)!}\Int_T \rho^{2k}(t) \ud t.		
		\end{gather*}
\end{enumerate}		
 \end{proof}

\begin{proof}[Proof of Corollary \ref{cor1}]
Choose $\mu=\delta_{\{ u_0\}}$, $u_0\in\R$ and write

\begin{multline*}
\int\limits_T\int\limits_{\R^2}\cov_X(t,u,v)\, \mu(du)\,\mu(dv)\,dt =   \int\limits_T
\cov\left( \bar{F}_Z\big(u_0/A\big),
\bar{F}_Z\big(u_0/A\big) \right)
 \,dt \\
 +   \int\limits_T
 \E \left[ \cov_Z(t,u_0/A,u_0/A)\right]  \,dt
\ge  \int\limits_T \var\left( \bar{F}_Z\big(u_0/A\big)\right)  \,dt = +\infty
\end{multline*}
since $Z\in \pa$, $\bar{F}_Z\big(u_0/A\big)$ is non-degenerate and bounded.
\end{proof}

\begin{proof}[Proof of Corollary \ref{cor:rvt0}]
Without loss of generality assume $Z, Y\in$ \pa. Then $Y^\alpha\in$ \pa , too, and the second term in \eqref{eq:covRandVol} is nonnegative. Denote
$$A_{u,v}(t)=\cov\left( \bar{F}_Z\big(u/Y_0\big),
\bar{F}_Z\big(v/Y_t\big) \right), \quad u,v\in\R_+, \, t\in T . $$
Since $Y\in$ \pa   \ and the function $\bar{F}_Z\big(u/\cdot\big)$ is bounded and nondecreasing for $u>0$ we get $A_{u,v}(t)\ge 0$ for all $u,v\in\R_+, \, t\in T .$
	Using the regular variation of the tail  of $Z_0$, the independence of $Y$ and $Z$ and Potter bound \cite[Proposition 2.6]{Resnick07} one can easily show that under the above assumptions on the integrability of $Y$ it holds
\begin{equation*}
A_{u,v}(t)\sim\
\bar{F}_Z(u)\bar{F}_Z(v) \cov\left( Y_0^\alpha, Y_t^\alpha \right), \quad u,v\to+\infty,
\end{equation*}	
for any $t\in T$. Then  for sufficiently large $N>0$ there exists $u_0>N$ such that for the Dirac measure $\mu=\delta_{\{  u_0 \}}$ and some $\varepsilon\in(0,1)$ we have
\begin{equation*}
\int\limits_T \int\limits_{\R^2}\! \!
\cov_X(t,u,v)
 \mu(du)\mu(dv)dt\ge
 \int\limits_T \! \!  A_{u_0,u_0}(t) dt\ge \\
 \varepsilon  \bar{F}_{Z}^2 (u_0) \! \! \int\limits_T \! \!  \cov\left( Y_0^\alpha, Y_t^\alpha \right) dt
\end{equation*}
which is infinite if $Y^\alpha$ is l.r.d.
Thus, $X=Y Z$ is l.r.d. if $Y^\alpha$ is l.r.d.
\end{proof}

\begin{proof}[Proof of Lemma \ref{lemm:DSubGauss}]
Without loss of generality, assume $G$ to be nonnegative.  By Lemma \ref{lem:CovGauss},  Fubini and Tonelli theorems for $G_u(y)=\bar{F}_Z\left(u/G(y)\right)$  we get
\begin{multline*}
D_{\mu}\left(G(Y),Z_0\right)=\int_T\int_{\R^2} \cov \left( G_u(Y_0), G_v(Y_t)\right) \mu(du) \mu(dv) \, dt \\
= \sum\limits_{k=1}^\infty \frac{\left(  \int_\R \langle G_u,  H_k\rangle_\varphi \, \mu(du) \right)^2 }{k!}\int_T \rho^k(t) \, dt .
\end{multline*}
The change of order of the sum and integrals is justified by Weierstrass uniform convergence test since for almost all $t\in T$
$$
\sum\limits_{k=1}^\infty \frac{\left| \langle G_u,  H_k\rangle_\varphi \langle G_v,  H_k\rangle_\varphi \right| }{k!}\rho^k(t) \le
\sum\limits_{k=1}^\infty \frac{  \langle 1,  |H_k|\rangle_\varphi^2 }{k!}\rho^k(t) \le
\sum\limits_{k=1}^\infty \rho^k(t) <\infty
$$
due to $\langle 1,  |H_k|\rangle_\varphi \le \sqrt{k!}$ by Cauchy--Schwarz inequality and due to condition  {\bf ($\rho$)}.
\end{proof}

\begin{proof}[Proof of Lemma \ref{lemm:HermRank}]
\begin{enumerate}
\item If $G:\R\to\R $ is monotone then $\rg (G)=1$ due to
\begin{equation}\label{eq:GH1}
\langle G  , H_1\rangle_\varphi=\E[Y G(Y)] = \int_0^\infty \left( G(y)- G(-y)\right)y \varphi(y)dy    \neq 0.
\end{equation}
What is the Hermite rank of $\zeta_{G,Z,u} $? First consider $Z\equiv 1$. Since the Hermite rank of
$y\mapsto \ind \{ y>u \} - \bar{F}_Y(u)$ is one we can write
$$
\langle \zeta_{G,1,u}  , H_1\rangle_\varphi=\E[Y \ind \{  G(Y) >u \}] = \E[Y \ind \{  Y >G^-(u) \}]   \neq 0,
$$
where  $G$ is non--decreasing w.l.o.g. Hence, $\rg(  \zeta_{G,1,u})=1 $ for any $u\in \R$.
Now let $G:\R\to\R_\pm$ and $Z$ be arbitrary. W.l.o.g. assume $G$ to be nonnegative.
Then
$$
\langle \zeta_{G,Z,u} , H_1\rangle_\varphi=\int_\R \bar{F}_Z\big(  u/G(y)\big) y \varphi (y)\, dy \neq 0,
$$
since for any $u\neq 0$ the function $y\mapsto \bar{F}_Z\left(u/G(y)\right)$ is monotone, and we can use the reasoning \eqref{eq:GH1}.  For nonpositive $G$ replace $ \bar{F}_Z$ above by $ {F}_Z$.

\item W.l.o.g. assume that $G$ is nonnegative and nondecreasing. We prove that $\rg (\widetilde{G})=2$.

Clearly, since $y\mapsto G(|y|)$ is even, we have $\E[Y G(|Y|)]=0$. Now,
\begin{align*}
\E[H_2(Y)G(|Y|)]=2\int_0^\infty G(y)(y^2-1)\varphi(y)dy\;.
\end{align*}
We note that
\begin{equation}\label{eq:int}
\int_0^\infty (y^2-1)\varphi(y)dy=0
\end{equation}
 and hence {by symmetry} $\int_0^1(y^2-1)\varphi(y)dy=-\int_1^\infty(y^2-1)\varphi(y)dy$.
Also, by the mean value theorem, due to monotonicity of non--constant $G$, there exists $y_0\in [0,1)$ such that
\begin{align*}
\int_0^1 G(y)(y^2-1)\varphi(y)dy=G(y_0)\int_0^1(y^2-1)\varphi(y)dy\;.
\end{align*}
Therefore,
\begin{align*}
\int_0^\infty G(y)(y^2&-1)\varphi(y)dy \\
&\geq G(y_0) \int_0^1 (y^2-1)\varphi(y)dy + G(1)\int_1^\infty (y^2-1)\varphi(y)dy\\
& = -G(y_0) \int_1^\infty (y^2-1)\varphi(y)dy +G(1)\int_1^\infty (y^2-1)\varphi(y)dy\\
& =(G(1)-G(y_0)) \int_1^\infty (y^2-1)\varphi(y)dy>0\;.
\end{align*}
For nonnegative nonincreasing $G$, we can use the estimate
\begin{align*}
&\int_0^\infty G(y)(y^2-1)\varphi(y)dy \leq G(y_0) \int_0^1 (y^2-1)\varphi(y)dy \\
& + G(1)\int_1^\infty (y^2-1)\varphi(y)dy
 =(G(y_0)-G(1)) \int_0^1 (y^2-1)\varphi(y)dy<0\;.
\end{align*}
If $G(y)\le 0$ just multiply it by $-1$. This proves that the Hermite rank of $G(|y|)$ is 2.

Now compute the Hermite rank of $\zeta_{\widetilde{G},1,u} $ for any $u\in\R$. Since  $\zeta_{\widetilde{G},1,u} $
is even, $\rg (\zeta_{\widetilde{G},1,u})>1$. Assuming w.l.o.g. that $G$ is nonnegative and nondecreasing we calculate
\begin{align*}
\langle & \zeta_{\widetilde{G},1,u}, H_2 \rangle_\varphi=\E[(Y^2-1)\ind \{G(|Y|)>u\}] \\
&= \int_{\R}(y^2-1)\ind \{|y|>G^-(u)\} \varphi(y)\, dy=2\int_{G^-(u)}^\infty (y^2-1) \varphi(y)\, dy\neq 0
\end{align*}
due to \eqref{eq:int} and $G^-(u)\neq 0$. So $\rg \zeta_{\widetilde{G},1,u}=2$. For general $Z$, we note that
$\zeta_{\widetilde{G},Z,u}$ is even, so $\rg (\zeta_{\widetilde{G},Z,u})>1$. If $G$ is non--negative then
$$
\langle \zeta_{\widetilde{G},Z,u} , H_2\rangle_\varphi=\int_\R \bar{F}_Z\big(  u/G(|y|)\big) H_{2}(y) \varphi (y)\, dy\neq 0
$$
by the first part of the proof of 2)  since $\bar{F}_Z\big(  u/G(|y|)\big)$ is a monotone even function of $y$.
Modifications of the proof for $G\le 0 $ or $G$ nonincreasing are obvious.
\end{enumerate}
\end{proof}
\begin{proof}[Proof of Theorem \ref{thm:LTrvf}]
Let ${\cal Y}$ be the $\sigma$--algebra  generated by the entire random field $\{Y_t,t\in \Z^d\}$. Then
\begin{equation*}
\sum_{t\in W_n}g(X_t) =\sum_{t\in W_n}\left(g(X_t)-\E[g(X_t)\mid {\cal Y}]\right)+\sum_{t\in W_n}\E[g(X_t)\mid {\cal Y}]= M_n+K_n\;,
\end{equation*}
where
 $$M_n=\sum_{t\in W_n}\left(g(X_t)-\E[g(X_t)\mid {\cal Y}]\right)=\sum_{t\in W_n} m(Y_t,Z_t)$$ and
$$K_n=\sum_{t\in W_n}\E[g(X_t)\mid {\cal Y}] =\sum_{t\in W_n}\xi(Y_t) \;.$$
The above decomposition is allowed by \eqref{eq:clt-cond-1}.
The limiting behaviour of the sum depends on an interplay between $M_n$ and $K_n$.
First, we state the limiting results for $M_n$ and $K_n$ separately.
\begin{lemma}\label{lem:iid}
Under the assumptions of Theorem \ref{thm:LTrvf}, it holds
\begin{align*}
\tilde M_n:=n^{-d/2}M_n\tod {\cal N}(0,\sigma^2)\;,
\end{align*}
where $\sigma^2=\E[\chi(Y_0)]2^d>0$.
\end{lemma}
\begin{proof}{
We calculate
\begin{align*}
&\E\left[\exp\{iz \tilde M_n\}\mid {\cal Y}\right] =
\E\left[\exp\left\{\frac{iz}{n^{d/2}} \sum_{t\in W_n} m(Y_t,Z_t)\right\}\mid {\cal Y}\right] \\
&
=:\E\left[\exp\left\{\frac{iz}{n^{d/2}} \sum_{t\in W_n}V_{t}\right\}\mid {\cal Y}\right]\;,
\end{align*}
where
$
V_{t}=m(Y_t,Z_t).
$
Note that, due to stationarity of $Y$ and $Z$, the random variables $V_{t}$ are identically distributed and conditionally independent, given ${\cal Y}$. Therefore,
\begin{align*}
&\E\left[\exp\{iz \tilde M_n\}\mid {\cal Y}\right] =
\E\left[\exp\left\{\frac{iz}{n^{d/2}} \sum_{t\in W_n}V_{t}\right\}\mid {\cal Y}\right]\\
&= \prod_{t\in W_n}\E\left[\exp\left\{\frac{iz}{n^{d/2}} V_{t}\right\}\mid {\cal Y}\right]\;.
\end{align*}
The standard inequality,
\begin{align*}
|\exp(it z)-(1+itz-t^2z^2/2)|\leq \min\{|tz|^2,|tz|^3\}
\end{align*}
yields
\begin{align*}
&\left|\E\left[\exp\left\{\frac{iz}{n^{d/2}} V_{t}\right\}\mid {\mathcal Y}\right]-\E\left[\left(1+\frac{izV_t}{n^{d/2}}-\frac{1}{2}\frac{z^2V_t}{n^d}\right)\mid {\mathcal Y}\right]\right|\\
&\leq \E\left[\min\left\{\frac{|z|^2V_t^2}{n^d},
\frac{|z|^3|V_t|^3}{n^{3d/2}} \right\}\mid {\mathcal Y}\right]=:\E[V_{t,n}\mid {\mathcal Y}]\;.
\end{align*}
For complex numbers $z_1,\ldots,z_m$, $w_1,\ldots,w_m$ of modulus at most 1, we have
\begin{align*}
\left|\prod_{i=1}^mz_i-\prod_{i=1}^mw_i\right|\leq \sum_{i=1}^m|z_i-w_i|\;.
\end{align*}
Hence
\begin{align*}
&A_n({\mathcal Y}):=\left|\prod_{t\in W_n}\E\left[\exp\left\{\frac{iz}{n^{d/2}} V_{t}\right\}\mid {\cal Y}\right]-
\prod_{t\in W_n}
\E\left[\left(1+\frac{izV_t}{n^{d/2}}-\frac{1}{2}\frac{z^2V_t^2}{n^d}\right)\mid {\mathcal Y}\right]
\right|\\
&\leq \sum_{t\in W_n}\left|\E\left[\exp\left\{\frac{iz}{n^{d/2}} V_{t}\right\}\mid {\cal Y}\right]-\E\left[\left(1+\frac{izV_t}{n^{d/2}}-\frac{1}{2}\frac{z^2V_t^2}{n^d}\right)\mid {\mathcal Y}\right]\right|\\
&\leq \sum_{t\in W_n}\E[V_{t,n}\mid {\mathcal Y}]\;.
\end{align*}
We argue that
\begin{align}\label{eq:AnY-tozero}
A_n({\mathcal Y})\to 0
\end{align}
in probability.
If this is the case, then the conditional characteristic function
\begin{align*}
\E\left[\exp\{iz \tilde M_n\}\mid {\cal Y}\right]
\end{align*}
and
\begin{align*}
B_n({\mathcal Y}):=\prod_{t\in W_n}
\E\left[\left(1+\frac{izV_t}{n^{d/2}}-\frac{1}{2}\frac{z^2V_t^2}{n^d}\right)\mid {\mathcal Y}\right]
\end{align*}
have the same limit in probability.  Applying the $\log$ to the above expression and $\log(1-x)=-x+O(x^3)$ we have
\begin{align*}
&\log B_n({\mathcal Y})=
\sum_{t\in W_n}\log\E\left[1+\frac{iz V_{t}}{n^{d/2}}- \frac{z^2 V_{t}^2}{2 n^{d}}\mid {\cal Y}\right]\\
&  =\frac{iz}{n^{d/2}}\sum_{t\in W_n}\E[V_t\mid {\mathcal Y}]-\frac{z^2}{2n^{d}}\sum_{t\in W_n} \E[V_{t}^2\mid {\cal Y}]\\
&+ O(1)\frac{|z|^3}{n^{3d/2}}\sum_{t\in W_n}\left(|\E[V_t\mid {\mathcal Y}]|\right)^3+O(1)\frac{z^6}{n^{3d}}\sum_{t\in W_n} \left(\E[V_{t}^2\mid {\cal Y}]\right)^3\;.
\end{align*}
The expression in the last line is $o_P(1)$ by \eqref{eq:clt-cond-2}.
By the definition, $\E[m(y,Z_t)]=0$ and hence $\E[V_{t}\mid {\cal Y}]=0$.
We have $\E[V_{t}^2\mid {\cal Y}]=\chi(Y_t)$
and therefore
\begin{align*}
&\log B_n ({\cal Y}) = -\frac{z^2}{2 n^{d}}\sum_{t\in W_n} \chi(Y_t)+o_p(1)\;.
\end{align*}
Since $\chi$ is measurable, the ergodic theorem (\cite[p. 339]{Yaglom871})
implies that
\begin{align*}
\frac{1}{n^{d}}\sum_{t\in W_n} \chi(Y_t) \stackrel{P}{\longrightarrow}
\E[\chi(Y_0)] 2^d, \quad n\to +\infty,
\end{align*}
whenever the covariance of the field $\chi(Y_t)$ goes to zero as $\|t\|\to +\infty$.
To check the latter property, we use Lemma \ref{lem:CovGauss} to conclude
\begin{align*}
|\cov(\chi(Y_0),\chi(Y_t))|\leq |C_Y(t)|\sum_{k=1}^\infty  \frac{\langle\chi, H_k \rangle_{\varphi}^2}{k!}\to 0
\end{align*}
as $\|t\|\to +\infty$, since the infinite series in the last expression is finite due to $\var(\chi(Y_0))<\infty$; cf. \eqref{eq:clt-cond-2}.}
Hence, $\log B_n({\cal Y})\to -z^2\sigma^2/2$ in probability.
By continuous mapping theorem, it holds
\begin{align*}
& \E\left[\exp\{iz \tilde{M}_n\}\mid {\cal Y}\right] \stackrel{P}{\longrightarrow}  e^{-z^2 \sigma^2/2}, \quad n\to +\infty\;.
\end{align*}
Since $\left|  \E\left[\exp\{iz \tilde{M}_n\}\mid {\cal Y}\right] \right|\le 1$ for all $n\in\N$ this sequence is uniformly integrable. Using the property of $L^1$--convergence of uniformly integrable sequences we get
$$
\E\left[\exp\{iz \tilde{M}_n\}\right]  \to e^{-z^2 \sigma^2/2}, \quad n\to +\infty,
$$
and we are done.
\end{proof}
\begin{lemma}\label{lem:lrd}
Under the assumptions of Theorem \ref{thm:LTrvf}, it holds
$$
n^{q\eta/2-d}L^{-q/2}(n)K_n\tod R\;, \quad n\to\infty.
$$
\end{lemma}
\begin{proof}{
Consider the random variable
$$
K_n(q)=\sum_{m=q}^{\infty}\frac{J(m)}{m!}\int_{[-n,n]^d}H_m(Y_t)dt\;.
$$
According to \cite[Theorem 4]{ LeonOlenko14} and \cite[Theorem 4.3]{AloOle19} the random variables
$$
\frac{K_n}{\sqrt{\var K_n}}
, \qquad
\frac{K_n(q)}{\sqrt{\var K_n(q)}}
$$
have the same limiting distributions as $n\to+\infty$.
Furthermore, if $\eta\in (0,d/q)$ we have by \cite[Theorem 5]{ LeonOlenko14} that
$$
n^{q\eta/2-d}L^{-q/2}(n)\int_{[-n,n]^d}H_q(Y_t)dt
$$
converges in distribution to  random variable $R$.
}
\end{proof}

If $\xi(y)\equiv 0$, the long memory part $K_n$ is not present and we apply Lemma \ref{lem:iid}. If $\xi(y)\not\equiv 0$, we note that the rate of convergence in Lemma \ref{lem:lrd} is slower than in Lemma \ref{lem:iid}, whenever $\eta\in (0,d/q)$.
\end{proof}

\section*{Acknowledgement}
We thank P. Doukhan for his remarks on $\psi$--mixing. E. Spodarev is grateful to the German Academic Exchange Service (DAAD) for the support of his research  stay in Ottawa in the fall 2015.



\newcommand{\noopsort}[1]{} \newcommand{\printfirst}[2]{#1}
  \newcommand{\singleletter}[1]{#1} \newcommand{\switchargs}[2]{#2#1}
\providecommand{\bysame}{\leavevmode\hbox to3em{\hrulefill}\thinspace}
\providecommand{\MR}{\relax\ifhmode\unskip\space\fi MR }
\providecommand{\MRhref}[2]{%
  \href{http://www.ams.org/mathscinet-getitem?mr=#1}{#2}
}
\providecommand{\href}[2]{#2}

\end{document}